\documentclass[11pt]{amsart}
\usepackage{lingmacros}
\usepackage{tree-dvips}

\usepackage{latexsym,amsmath,amssymb,amsfonts,amscd,graphics,appendix,amsxtra,amsthm}
\usepackage[mathscr]{eucal}
\usepackage[all]{xypic}
\usepackage[normalem]{ulem}
\usepackage{indentfirst}

\usepackage{mathrsfs}

\usepackage{tikz-cd}

\numberwithin{equation}{section}
\newtheorem{theorem}[equation]{Theorem}
\newtheorem{lemma}[equation]{Lemma}

\newtheorem{proposition}[equation]{Proposition}
\newtheorem{corollary}[equation]{Corollary}

\def\Hom{\operatorname{Hom}}

\newcommand{\bA}{\mathbb{A}}
\newcommand{\bH}{\mathbb{H}}
\newcommand{\bR}{\mathbb{R}}

\newcommand{\bQ}{\mathbb{Q}}
\newcommand{\bZ}{\mathbb{Z}}

\newcommand{\bC}{\mathbb{C}}

\newcommand{\GL}{\mathrm{GL}}

\newcommand{\SL}{\mathrm{SL}}
\newcommand{\Sp}{\mathrm{Sp}}

\newcommand{\SO}{\mathrm{SO}}

\title[Refined Gan-Gross-Prasad conjecture for cusp forms of GSp(4)]{On the refined Gan-Gross-Prasad conjecture for cusp forms of GSp(4)}
\author{Jun Wen}   \email{jwen@math.umass.edu} \address{Department of Mathematics and Statistics, University of Massachusetts Amherst}

\begin{document}
\bibliographystyle{alpha}
\maketitle

\begin{abstract}
We prove a conjectural formula relating the Bessel period of certain automorphic forms on $\mathrm{GSp}_4$ to a central $L$-value. This formula is proposed by Liu \cite{liu} as the refined Gan-Gross-Prasad 
conjecture for the groups $(\SO(5), \SO(2))$. The conjecture has been previously proved for certain automorphic forms on $\mathrm{GSp_4}$ from lifts.  In this paper, we extend the formula to Siegel modular forms 
of $\Sp_4(\bZ)$.
\end{abstract}

\section{Overview}
\subsection{Introduction} 
In this paper, as a sequel of \cite{junwen15}, we extend the Waldspurger's theorem on the toric period of automorphic forms on $\GL_2$ to the Bessel period of automorphic forms of $\mathrm{GSp}_4$. In their original form,  Gross and Prasad in \cite{grp1} and \cite{grp2} gave a series of conjectures on the restriction of an automorphic representation of a special 
orthogonal group to smaller special orthogonal group. Later on in \cite{gangrossprasad} the conjecture was extended to all classical groups. Recently Ichino and Ikeda \cite{ichinoikeda} formulated a refinement of the
Gan-Gross-Prasad conjecture for the orthogonal groups $\SO_{n+1} \times \SO_n$. 

There is much progress towards the global case of Gan-Gross-Prasad conjecture on the unitary groups. Using the relative trace formula developed by Jacquet-Rallis \cite{jacquetrallis}, W.Zhang proved the conjecture
for the unitary groups $\mathrm{U}_{n+1} \times \mathrm{U}_n$ under some local conditions \cite{zhangwei}.
Although the Jacquet-Rallis trace formula has been successful for the unitary groups, there is still no strategy works in all cases for the orthogonal groups $\SO_{n+1} \times \SO_n$. 
The $n=2$ case is the original Waldspurger's formula \cite{waldspurger} and $n=3$ case corresponds to the triple product formula \cite{ichino}. Gan and Ichino proved the case of $n=4$ in \cite{ganichino} still using the theta correspondence. But little is known beyond this. 

A recent work of Liu \cite{liu} offered a precise conjectural formula for the Bessel period of automorphic forms in terms of the central values of certain $L$-functions.
Liu gave a proof of the conjecture in the special case of Yoshida lifts. The proof of in the case of non-endoscopic Yoshida lifts has been given by \cite{corbett}.

In \cite{junwen15} the author developed a new method to prove the classical Waldspurger's theorem: the key point is to exploit the distributions on certain prehomogeneous vector spaces acted by $\GL_2$.
In this paper, we consider the same prehomogeneous vector spaces with group action replaced by $\mathrm{GSp}_4$. 

\subsection{Statement of the result}
Let $F$ be a number field with ad$\grave{\mathrm{e}}$le ring $\bA = \bA_F$. Let $W$ be the standard symplectic space over $F$ of dimension $4$. Let $W_1 = \mathrm{Span}\{e_1, e_2\}$ and $W_1^{\vee} = \mathrm{Span}\{e^{\vee}_1, e^{\vee}_2\}$ which form a complete polarization with $\langle e_i, e^{\vee}_j \rangle = \delta_{ij}$. Let $G = \mathrm{GSp}_4$. Denote by $P$ the Siegel parabolic subgroup stabilizing $W_1$, which has the Levi
decomposition $P = M U$. Then $M \cong \GL_2 \times F^{\times}$ and $U \cong \mathrm{M}_2^{\mathrm{Sym}}$. For $S \in \mathrm{M}_2^{\mathrm{Sym}}$ with $\det(S) \neq 0$, define
\begin{align*}
T = \{ (g, \det(g)) \in M | {}^t g S g = \det(g) S \}.
\end{align*}
Then $S$ imposes a non-degenerate quadratic form $q_S$ on $W_1$. We have
\begin{align*}
T = \mathrm{GSO}(W_1) \cong \mathrm{Res}^K_F K^{\times},
\end{align*}
where $K$ is the discriminant quadratic algebra. Define the Bessel subgroup $R = U \rtimes T$. 

Let $\chi$ be the unitary Hecke character of $\bA_{K}^{\times}$, and let $\pi$ be an irreducible, cuspidal automorphic representation of $\mathrm{GSp}_4(\bA)$ in the space of cusp forms $\mathcal{V}_{\pi}$. Impose the central character condition $\pi \boxtimes \chi|_{\bA^{\times}} = 1$. Given an additive character $\psi: F \backslash \bA \to \bC^{\times}$, for an element $S  \in \mathrm{M}_2^{\mathrm{Sym}}$, we define an automorphic character $\theta_S$ of $N$ by the formula
\begin{align*}
\theta_S(X) =  \psi (\mathrm{tr}(SX)) : U(F) \backslash U(\bA) \to \bC^{\times}. 
\end{align*} 
We then define the Bessel period of $f \in \pi$ to be the integral
\begin{align*}
\mathcal{P}(f, \chi) = \int_{\bA^{\times} T(\bQ) \backslash T(\bA) } \int_{U(\bQ) \backslash U(\bA) } f(ug) \chi(g) \theta_S^{-1}(u) du dg. 
\end{align*}  
This integral defines an element of $\Hom_{R(\bA)} (\pi\boxtimes \chi \otimes \theta_S^{-1} , \bC)$. Put $\mathcal{V}_{\bar{\pi}} = \{ \bar{f} | f \in \mathcal{V}_{\pi} \}$, which forms the space of contragredient representation $\bar{\pi}$ of $\pi$. Then we have a canonical pairing 
\begin{align*}
\mathcal{B}_{\pi}: \mathcal{V}_{\pi} \otimes \mathcal{V}_{\bar{\pi}} \to \bC
\end{align*}
defined by the Petersson inner product 
\begin{align*}
\mathcal{B}_{\pi}(f_1, f_2) =\int_{Z(\bA) G(F) \backslash G(\bA) } f_1(g) f_2(g) dg. 
\end{align*}

Now suppose that $\pi$ decompose as $\otimes_v \pi_v$ into admissible representations over local field. Assume $f = \otimes f_{v} \in \mathcal{V}_{\pi}$ and 
$\tilde{f} = \otimes \tilde{f}_{v} \in \mathcal{V}_{\bar{\pi}}$ , then for each place $v$ of $F$, we define the integral of matrix coefficients
\begin{align*}
\alpha(f_v, \tilde{f}_{v}; \chi) = \int_{F^{\times} \backslash T} \int_{U} \mathcal{B}_{\pi, v} (\pi(gu) f_v, \tilde{f}_v) \chi(g)  \theta_S^{-1}(u)   dg du. 
\end{align*}
In his paper, Liu \cite[Theorem 2.2]{liu} shows there exists a set of good places, for which the local integral is given as follows
\begin{align*}
\alpha(f_v, \tilde{f}_{v}; \chi) = \frac{\zeta_v(2) \zeta_v(4) L_v(\frac{1}{2}, \pi \boxtimes \chi)}{L_v(1, \pi, \mathrm{Ad}) L_v(1, \chi_{K/F})}. 
\end{align*}
Define the normalized local integral of matrix coefficients by
\begin{align*}
\alpha^{\natural}(f_{v}, \tilde{f}_{v}; \chi) = \left(\frac{\zeta_v(2) \zeta_v(4)L_v(\frac{1}{2}, \pi)}{L_v(1, \pi, \mathrm{Ad}) L_v(1, \chi_{K/F})} \right)^{-1}\alpha(f_{v}, \tilde{f}_{v}; \chi).
\end{align*}
Then our main result is the following

\begin{theorem}\label{main}
Let $F = \bQ$ and let the unitary Hecke character $\chi$ of $\bA_K^{\times}$ be trivial. Let $f$ be the adelization of a Siegel modular form of $\mathrm{Sp}_4(\bZ)$, then
\begin{align*}
\frac{|\mathcal{P}(f, \chi)|^2}{||f||^2 \mathrm{vol}(K^{\times} \bA^{\times} \backslash \bA^{\times}_K) }
 =  \frac{1}{8}  \frac{\zeta(2) \zeta(4)L(\frac{1}{2}, \pi)}{L(1,  \pi, \mathrm{Ad}) L(1, \chi_{K/F})} \prod_v \alpha^{\natural}(f_{v}, \tilde{f}_{v}; \chi).
\end{align*}
\end{theorem}

It should be mentioned our result provides a proof of B$\ddot{\mathrm{o}}$cherer's conjecture \cite{bocherer} with trivial Hecke character over rational field. Recall the B$\ddot{\mathrm{o}}$cherer's conjecture formulates an equality between 
sums of Fourier coefficients of Siegle modular forms and certain $L$-values. If automorphic form in question is the adelizaton of Siegel modular forms, then the Bessel period turns out to be precisely the Fourier coefficients that B$\ddot{\mathrm{o}}$cherer considered.  

\subsection{Idea of the proof}
In \cite{junwen15}, the author generalized the Shintani's zeta functions to the higher rank automorphic forms. 
The Shintani's zeta functions associated to prehomogeneous vector spaces (PVS) was first introduced in \cite{satoshintani}. 
They are usually defined as Dirichlect series in one variable with coefficients encoding the class numbers of certain field extensions. 
Studying the functional equations of these zeta functions, Shintani in \cite{shintani}, \cite{shintani2} obtained the average value of class number of binary quadratic forms and binary cubic forms respectively. 
The key tool in Shintani's proof of the analytic continuation of these zeta functions is the Poisson summation formula. 
Later on, the adelization of Shintani's zeta function was treated in a few cases (\cite{wright}, \cite{wright2}). 
The recent breakthroughs of M.Bhargava in counting quartic and quintic number fields came from a detailed study of integral orbits of certain PVS. In the first of his series paper on ``Higher composition laws" \cite{bhargava1}, \cite{bhargava2},
\cite{bhargava3} and \cite{bhargava4} he focused on the space $V_{\bZ} = \bZ^2 \otimes \bZ^2 \otimes \bZ^2$ acted on by $\SL_2(\bZ) \times \SL_2(\bZ) \times \SL_2(\bZ)$. Inspired by the Shintani's function in two variables studied in \cite{shintani}, we 
in \cite{junwen14} considered the action of the subgroup $B_2(\bZ) \times B_2(\bZ) \times \SL_2(\bZ)$, where $B_2(\bZ)$ denotes the group of integral upper-triangle matrices of $\SL_2(\bZ)$. It was shown that the zeta function constructed has three variables and coincides with the $A_3$ quadratic Weyl group multiple Dirichlect series. In \cite{junwen15}, we studied the pairs of quaternary alternating $2$-forms $\bZ^2 \otimes \wedge^2 \bZ^4$ acted by
$\SL_2(\bZ) \times \SL_4(\bZ)$, that is obtained from $\bZ^2 \otimes \bZ^2 \otimes \bZ^2$ via the operation of skew-symmetrizing. Then we define distributions associated to respective PVS. The one associated to the PVS of $2\times 2 \times 2$ matrices gives rise to the product of the sum of automorphic form valued at the Heegner points. The other one associated to the PVS of quaternary alternating $2$-forms gives rise to the central value of base change L-function of the automorphic form. The two distributions are related via the Poisson summation formula.  In this paper we will apply the same strategy.

The first step of our proof is to realize the period integrals of an automorphic form as the distribution associated to a certain PVS. 
Inspired by the Bhargava's work on the extension of Gauss's composition law, we consider the space of $2 \times 2\times 2$ boxes.
One of Bhargava's achievements is the determination of the corresponding integral orbits, i.e. the determination of the $\SL_2(\bZ)^3$-orbits on $\bZ^2 \otimes \bZ^2 \otimes \bZ^2$. In particular, he discovered that in this case, the generic integral orbits are in bijection with isomorphism classes of tuples $(A, I_1, I_2, I_3)$ where 
\begin{itemize}
\item[(a)] $A$ is an order in an \'{e}tale quadratic $\bQ$-algebra; 
\item[(b)] $I_1,I_2$ and $I_3$ are elements in the narrow class group of $A$ such that $I_1 \cdot I_2 \cdot I_3 =1$.
\end{itemize}
Forgetting the group structure, Bhargava's composition law states that its integral orbits are in one to one correspondent to the pairs of (strict) class groups of binary quadratic forms. This motivates us to define a distribution on the 
real space of $2 \times 2\times 2$ boxes and show it gives rise to the product of the sum of automorphic form valued at the Heegner points. 

The next step is to relate this integral to the special values of L-functions of automorphic forms. The key point here is to study a certain parabolic group action on the  skew-symmetrization of $2 \times 2\times 2$ boxes, i.e., the PVS of pairs of quaternary alternating $2$-forms.  We will also define a distribution on this PVS and apply the Poisson summation formula to show there is a simple pole of the distribution and its residue gives rise to the distribution on the PVS of  $2 \times 2\times 2$ boxes defined earlier. In some sense, the set of $2 \times 2\times 2$ boxes can be viewed as the singular set of  PVS of pairs of quaternary alternating $2$-forms.

The last step is to study the adelic version of the distribution on PVS of pairs of quaternary alternating $2$-forms. The reason for this consideration is that there is only one rational orbit of this PVS; moreover, the rational stabilizer group of 
a representative is the unipotent subgroup of $\GL_2$ and the adelic stabilizer group is a $\GL_2$. This means integrating along the rational stabilizer group, after the Rankin-Selberg method, becomes an Euler product. After carefully analyzing the local orbit integrals at unramified places via the local orbits counting, we show each of these local integrals is related to the Fourier coefficients of matrix coefficients. Thanks to the work of Ichino-Ikeda \cite[Theorem 1.2]{ichinoikeda} and Liu \cite[Theorem 2.2]{liu} on the study of Whittaker-Shintani functions, the Fourier coefficients of matrix coefficients at non-archimedean places provide the local products of central value of L-function of $\mathrm{GSp}_4$.


\section{Introduction}
Let $F$ be a number field with ad$\grave{\mathrm{e}}$le ring $\bA =\bA_F$. Let $W = F^4$ and endow $W$ with an antisymmetric bilinear form so that $W$ becomes a symplectic vector space over $F$.
Recall the group $\mathrm{GSp}_4(F)$ is defined by
\begin{align*}
\mathrm{GSp}_4(F) = \left\{ g \in \GL(W) : (gu, gv) = \mu(g) (u, v),  u, v \in W \right\}.
\end{align*}
Let $W_1 = F^2$ then $W= W_1 \oplus W_1^{\vee}$.

Throughout this paper we assume $F = \bQ$. Let $f$ be a Siegel cusp form of degree 2 and weight $k$ for the group $\Gamma_2 = \Sp_4(\bZ)$. Hence for all $\gamma \in \mathrm{Sp}_4(\bZ)$, we have $f|_{k}\gamma = f$, where
\begin{align*}
(f|_{k }\gamma)(Z) = \mu(g)^k j(g, Z)^{-k} f(gZ). 
\end{align*}
for $g \in \mathrm{GSp}_4(\bR)^{+}$ and $Z \in \bH_2$. 
Here $j(g, Z) = \det(CZ + D)$ for $g = \left(   \begin{matrix}  A & B \\
C & D \end{matrix}  \right) \in \mathrm{GSp}_4(\bR)^{+}$.
The Fourier expansion of $f$  is given by
\begin{align*}
f(Z) = \sum_{S \ \mathrm{positive}} a(f, S) e^{2 \pi i \mathrm{Tr}(SZ)},
\end{align*}
where the matrices $S$ are of the form
\begin{align*}
S = \left(  \begin{matrix}
a & b/2 \\
b/2 & c
\end{matrix} \right), \ \ a, b, c \in \bZ, \ \  a > 0, \ \ \mathrm{disc}(S) = b^2 - 4 ac < 0.
\end{align*}
For $S = \left(  \begin{matrix}
a & b/2 \\
b/2 & c
\end{matrix} \right) \in \rm{M}_2^{\mathrm{sym}} (\bQ)$ such that $b^2 - 4 ac  = D < 0$, put $K = \bQ(\sqrt{D})$, let $\mathrm{Cl}_K$ denote the ideal class group of $K$ and $h_K$ denote the class number of quadratic field extension.
Let $\chi_{K / \bQ}$ denote the quadratic character
associated to the extension $K / \bQ$.
We have the number of root of unity
\begin{align*}
w(D) = \begin{cases} 
4  \ \mathrm{if} \ D = 4,\\
6 \ \mathrm{if} \ D = 3,\\
2 \ \mathrm{otherwise}.
\end{cases}
\end{align*}
Let $\pi = \otimes \pi_v$ be an irreducible, unitary, cuspidal, automorphic representation of $\mathrm{GSp}_4(\bA)$ with trivial central character and $f \in  \pi_{\infty}$. 
Write $g \in \mathrm{GSp}_4(\bA)$ as $g = g_{\bQ} g_{\infty} g_0$ with $g_{\bQ} \in \mathrm{GSp}_4(\bQ), g_{\infty} \in \mathrm{GSp}_4(\bR)^{+}, g_0 \in K_0$, where $K_0 = \prod_{p < \infty} \mathrm{GSp}_4(\bZ_p)$.
Denote the adelization of $f$ by $\Phi_f : \mathrm{GSp}_4(\bA)\to \bC$ that is defined by
\begin{align*}
\Phi_f (g)  = (f|_kg_{\infty})(i 1_2).
\end{align*}

A maximal, non-split torus in $\GL_2(\bQ)$ is given by the orthogonal group 
\begin{align*}
T = \left\{ g \in \GL_2(\bQ) : {}^tgS g = \det(g) S \right\} = \mathrm{GSO}(W_1).
\end{align*}
In fact, 
\begin{align*}
T(\bQ)  = \left\{ x  + y  \left( \begin{matrix}
b/2 & c \\
-a & -b/2
\end{matrix} \right) : x, y \in \bQ \right \}^{\times}.
\end{align*}
Then the isomorphism $T(\bQ) \cong \bQ(\sqrt{D})^{\times}$ is given by
\begin{align*}
x  + y  \left( \begin{matrix}
b/2 & c \\
-a & -b/2
\end{matrix} \right) \mapsto x + y \frac{\sqrt{D}}{2}. 
\end{align*}
Let $U$ be the unipotent radical 
\begin{align*}
U  = \left\{   \left( \begin{matrix}
1_2 & X \\
0 & 1_2
\end{matrix} \right)  : X \in \mathrm{M}_2^{\mathrm{sym}}  \right\}. 
\end{align*}
There is an embedding $T \to \mathrm{GSp}(W)$ given by mapping $g \in T$ to 
\begin{align*}
 \left( \begin{matrix}
g  & \\
 & \det(g)\cdot  {}^tg^{-1}
\end{matrix} \right) \in \mathrm{GSp}(W).
\end{align*}
Define the Bessel group to be the semidirect product $R = U \rtimes T$. 

Let $P = MU$ be the Siegel parabolic subgroup of $\mathrm{GSp}_4$. By the Iwasawa decomposition, $\mathrm{GSp}_4(\bA) = U(\bA) M(\bA) K_{\infty} K_0$, where $K_{\infty} \cong U_2(\bR)$ is the maximal compact subgroup of $\mathrm{Sp}_4(\bR)$.  
Let $e$ be a character $\bQ \backslash \bA \to \bC^{\times}$ such that $e(x) = e^{ 2 \pi i x}$ for $x\in \bR$ and $e(x) =1$ for $x \in \bZ_p$. Let $\theta_S:  U(\bQ) \backslash U(\bA) \to \bC^{\times}$ be the character given by
\begin{align*}
\theta_S \left(   \left( \begin{matrix}
1_2 & X \\
0 & 1_2
\end{matrix} \right)  \right) = e(\mathrm{Tr}(SX)). 
\end{align*}
Let $g = u  m k_{\infty} k_0$, with $ u \in U(\bA), m \in M(\bA), k_{\infty} \in K_{\infty}$ and $k_0 \in K_0$.
Let $m= m_{\bQ} m_{\infty} m_0$, with $m_{\bQ} \in M(\bQ), m_{\infty} \in M^{+}(\bR)$ and $m_0 \in M(\bA) \cap K_0$. 
Denote $Z = m_{\infty} \langle i 1_2 \rangle$. Let $S \in \mathrm{M}_2^{\mathrm{sym}}(\bQ)$ be non-degenerate, and let $S' = v {}^t A S A$.  We define the Fourier coefficient of $\Phi_f$, 
\begin{align*}
\Phi_f^S(g) = \int_{U(\bQ) \backslash U(\bA) } \Phi_f(ug) \theta_S^{-1}(u)  du \ \ \mathrm{for} \ g \in \mathrm{GSp}_4(\bA).
\end{align*}
Then, according to \cite[Proposition 3.1]{pitalesahaschmidt}, 
\begin{align*}
\Phi_f^{S}(g) = \theta_S(n_0)  \mu(m_{\infty}^k) j(g_{\infty}, i 1_2)^{-k} a(f, S') e^{2 \pi i \mathrm{Tr}(S'Z)}. 
\end{align*}
In fact, we have
\begin{lemma}
For $S \in \mathrm{M}_2^{\mathrm{sym}}(\bQ)$ and $S' = v {}^t A S A$, with $A \in \GL_2(\bQ)$ and $v \in \bQ^{\times}$, we have
\begin{align*}
\Phi_f^{S'}(g) = \Phi_f^S( \left(  \begin{matrix} A & \\
 & v^{-1} {}^t A^{-1}   \end{matrix}  \right)  g ) \ \ \mathrm{for} \ g \in \mathrm{GSp}_4(\bA).
\end{align*}
\end{lemma}

Let $\pi$ be an automorphic representation of $\mathrm{GSp}_4(\bA)$. Let $\chi$ be a unitary Hecke character of $\bA_K^{\times}$, that can be thought as the automorphic representation of the abelian group
$T(\bA)$ 
\begin{align*}
\chi: T(\bQ) \backslash T(\bA) \to \bC^{\times}.
\end{align*}
Impose the central character condition that $\pi\otimes \chi|_{\bA^{\times}} = 1$. 
For $f \in \pi$, the Bessel period of $f$ with respect to $\chi$ is defined to be the period integral 
\begin{align*}
\mathcal{P}(f, \chi) = \int_{\bA^{\times} T(\bQ) \backslash T(\bA) } \int_{U(\bQ) \backslash U(\bA) } f(ug) \chi(g) \theta_S^{-1}(u) du dg. 
\end{align*}  
This integral defines an element of $\Hom_{R(\bA)} (\pi\boxtimes \chi \otimes \theta_S^{-1} , \bC)$. 
The refined Gan-Gross-Prasad conjecture proposed by Liu states that 
\begin{align*}
|\mathcal{P}(f, \chi)|^2 \sim \frac{\zeta(2) \zeta(4) L\left( \frac{1}{2}, \pi\boxtimes \chi  \right)}{L\left( 1, \pi , \mathrm{Ad}\right) L(1, \chi_{K/F})  }.
\end{align*}

\section{Prehomogeneous vector spaces}
The proof relies on the construction of certain prehomogeneous vector spaces. Let $G$ be a connected complex Lie group, usually $G$ is a complexification of a real Lie group. A prehomogeneous vector space (PVS) $V$ of $G$, denoted by $(G, V)$, is a complex finite dimensional vector space $V$ together with a holomorphic representation of $G$, such that $G$ has an open dense orbit in $V$. Let $P$ be a complex polynomial function on $V$. We call it a relative invariant of $G$ if $P(g v) =\chi(g) P(v)$ for some rational character $\chi$ of $G$. 

\subsection{PVS of $2 \times 2 \times 2$ boxes}
Let  $V(\mathbb{\bZ})$ be the set of $2\times 2 \times 2$ integral cubes. For each element $A \in V(\bZ)$, there are three ways to form pairs of $2 \times 2$ matrices by taking the opposite sides out of $6$ sides. Denote them by
\begin{align}
&M_A^1=\left(\begin{array}{ccc}a &b \\
c &d
\end{array}\right); N_A^1= \left( \begin{array}{ccc}e &f \\
g &h
\end{array}\right),
\notag \\
&M_A^2=\left(\begin{array}{ccc}a &e \\
c &g
\end{array}\right); N_A^2 =\left( \begin{array}{ccc}b &f \\
d &h
\end{array}\right),
\notag \\
&M_A^3=\left(\begin{array}{ccc}a &e \\
b &f
\end{array}\right);  N_A^3=\left( \begin{array}{ccc}c &g \\
d &h
\end{array}\right).
\notag 
\end{align}
For each pair $(M_A^i, N_A^i)$ we can associate to it a binary quadratic form by taking
$$Q_A^i(u,v) =- \det(M_A^i u - N_A^i v) .$$
Explicitly for $A$ as above,
\begin{align} -Q_A^1(u,v) &= u^2(ad-bc) +uv ( -ah+bg+cf-de) +v^2(eh-fg), \notag \\
-Q_A^2 (u,v)&= u^2(ag-ce)+uv(-ah-bg+cf+de) + v^2(bh-df), \notag \\
-Q_A^3(u,v)&= u^2(af-be) +uv(-ah+bg-cf+de) +v^2(ch-dg).\notag \end{align}

The group $\SL_2(\bZ) \times \SL_2(\bZ) \times \SL_2(\bZ)$ acting on $V(\bZ)$ is defined by taking each $g_i$ in $(g_1, g_2, g_3)$ acting on the pair $(M_A^i, N_A^i), i  = 1,2, 3$.
For example, if $g_1 = \left( \begin{matrix} g_{11} & g_{12} \\
g_{21} & g_{22} \end{matrix} \right) $, then it acts on the pair $(M_A^1, N_A^1)$ by 
$$(M_A^1, N_A^1) \cdot \left( \begin{matrix} g_{11} & g_{12} \\
g_{21} & g_{22} \end{matrix} \right).$$

The action extends to the complex group $\GL_2(\bC) \times \GL_2(\bC) \times \GL_2(\bC)$ on the complexified vector space $V(\bC)$. This group action has only one relative invariant given by
\begin{align*}
P(A) & =\mathrm{disc}(Q^i_A)
\end{align*}
for $i=1, 2, 3$.
The corresponding rational character is
\begin{align*}
\chi(g) &=\det(g_1)^2 \det(g_2)^2 \det(g_3)^2.
\end{align*}
for $g = (g_1, g_2, g_3)$.
Furthermore, it is easy to show that
\begin{proposition}
The pair $\left(\GL_2(\bC) \times \GL_2(\bC) \times \GL_2(\bC) , V(\bC)\right)$ is a prehomogeneous vector space.
\end{proposition}

\subsection{PVS acted by maximal parabolic subgroup of $\mathrm{GSp}(4)$}
Let $P(\bQ)$ be a maximal parabolic subgroup of $\mathrm{GSp}_4(\bQ)$ which is the semiproduct $U(\bQ) \rtimes M(\bQ)$, where $U(\bQ)$ is defined before and 
\begin{align*}
M(\bQ) = \left \{    \left( \begin{matrix} g  & \\
 &  \det(g)\cdot {}^tg^{-1} \end{matrix} \right) : g \in \GL_2(\bQ)\right\} \cong \GL_2(\bQ).
\end{align*}
Throughout the paper we define the group acting on $V(\bQ)$ to be
\begin{align*}
G(\bQ) = \GL_2(\bQ) \times M(\bQ)  \times M(\bQ),
\end{align*}
where each copy of $M(\bQ)$ acts as the $\GL_2(\bQ)$ action defined before. Therefore
\begin{proposition}
The pair $\left(G(\bC) , V(\bC)\right)$ is a prehomogeneous vector space.
\end{proposition}

Let a negative integer $D$ be the fundamental discriminant. Let $V_D(\bQ)$ denote the subset of $V(\bQ)$ 
\begin{align*}
V_D(\bQ) = \left \{   P(A) = D : A \in V(\bQ)  \right \}.
\end{align*}
For $g \in G(\bQ)$, we will write $(g_1, g_2, g_3) \in  \GL_2(\bQ) \times M(\bQ)  \times M(\bQ)$ as the coordinate, and write 
\begin{align*}
g_2 =   \left(    \begin{matrix}   m(g_2)    & 0 \\
0 &\det(m(g_2)) \cdot {}^t m(g_2)^{-1} \end{matrix}\right).
\end{align*}
Denote by $G^1(\bQ)$ the subgroup of $G(\bQ)$
\begin{align*}
G^1(\bQ) = \left\{g \in G(\bQ)  : \det(g_1)\det(m(g_2)) \det(m(g_3)) = 1 \right\}.
\end{align*}
Then $G^1(\bQ)$ acts invariantly on $V_D(\bQ)$. Let $Z_G(\bQ)$ denote the center of $G(\bQ)$, then $Z_G^1(\bQ) = Z_G(\bQ) \cap G^1(\bQ)$ acts trivially on $V(\bQ)$.

\subsection{PVS of pairs of quaternary alternating $2$-forms}
Let $\bZ^2 \otimes \wedge^2 \bZ^4$ be the space of paris of quaternary alternating 2-forms. Denote its element $F= (M_F, N_F)$ by 
\begin{align*}
\left( \left( \begin{matrix}
 0 & r_1 & a & b \\
 -r_1 & 0 & c & d \\
 -a & -c & 0 & l_1 \\
 -b & -d & -l_1 & 0 
\end{matrix} \right),  \left( \begin{matrix}
 0 & r_2 & e & f \\
 -r_2 & 0 & g & h \\
 -e & -g & 0 & l_2 \\
 -f & -h & -l_2 & 0 
\end{matrix} \right) \right).
\end{align*}
The action of $\SL_2(\bZ) \times \SL_4(\bZ)$ is given as follows: an element $(g_1, g_2) \in \SL_2(\bZ)\times \SL_4(\bZ)$ acts by sending the pair $F=(M_F, N_F)$ to
 $$ (M_F, N_F) \cdot (g_1, g_2) = ( s \cdot {}^tg_2 M_F g_2 + u\cdot {}^tg_2 N_F g_2, t\cdot {}^tg_2 M_F g_2+ v\cdot {}^tg_2 N_F g_2) ,$$
where $g_1 = \left( \begin{array}{cc} s & t \notag \\ u & v \notag \end{array} \right)$. Then the fusion process is defined as a $\bZ$-linear mapping
\begin{align*}
\mathrm{id}\otimes \wedge_{2,2}: \bZ^2 \otimes \bZ^2 \otimes \bZ^2 \to \bZ^2 \otimes \wedge^2(\bZ^2 \oplus \bZ^2) = \bZ^2 \otimes \wedge^2\bZ^4.
\end{align*}
Explicitly, it is given by
\begin{align*}
 \left( \left( \begin{matrix}   
a & b \\
c & d
\end{matrix} \right),  \left( \begin{matrix}   
e & f \\
g & h
\end{matrix} \right) \right) \to \left( \left( \begin{matrix}
 0 & 0 & a & b \\
 0 & 0 & c & d \\
 -a & -c & 0 & 0 \\
 -b & -d & 0 & 0 
\end{matrix} \right),  \left( \begin{matrix}
 0 & 0 & e & f \\
 0 & 0 & g & h \\
 -e & -g & 0 & 0 \\
 -f & -h & 0 & 0 
\end{matrix} \right) \right).
\end{align*}
One of Bhargava's higher composition law proves the above mapping is surjective on the level of equivalent classes.

To each $F =(M_F, N_F)\in \bZ^2 \otimes \bZ^4$, we can associate a binary quadratic form $Q$ given by
\begin{equation*}
Q_F(u, v) = -\mathrm{Pfaff}(M_F u - N_F v) = -\sqrt{\mathrm{Det} (M_F u- N_Fv)},
\end{equation*}
where the sign of the $\text{Pfaff}$ is chosen by taking
$$\text{Pfaff}\left( \left( \begin{matrix}  & I  \\
-I &  \end{matrix} \right) \right) = 1, $$
or alternatively, 
$$\text{Pfaff}\left( \left( \begin{matrix}  0 & r & a & b  \\
-r & 0 & c & d \\
  -a & -c & 0 & l\\
  -b & -d & -l & 0 \end{matrix} \right) \right) = ad-bc - rl . $$
  
We now consider the subgroup of $\SL_2(\bZ)\times \SL_4(\bZ)$ defined by $
H(\bZ) = B_2(\bZ)  \times P_{2,2}(\bZ)
$, where $P_{2,2}(\bZ)$ is the upper triangular parabolic subgroup of $\SL_4(\bZ)$ that has the shape 
\begin{align*}
\left(   \begin{matrix} \ast & \ast& s & t\\
  \ast & \ast& u & v\\
  0 & 0 & \ast  & \ast \\
  0 & 0 & \ast & \ast \end{matrix}   \right)
\end{align*}
with the upper left $2 \times 2$ matrix determinant $1$.  

Let $W(\bC)$ be the subspace of elements with vanishing $r_1$-position. Then the action extends to the complex group $H(\bC)$ on complexified vector space 
$W(\bC)$. It has three relative invariants, they are 
\begin{align*}
\mathrm{disc}(F) &= \mathrm{disc} ( Q_F) , \\
P_0(F)& = r_2(F), \\
P_1(F) &= -\mathrm{Pfaff}(M_F) . 
\end{align*}
We denote by $\chi_0, \chi_1$ the characters of group $H(\bC)$ determined by the polynomial invariants $P_0, P_1$ respectively, i.e.,
\begin{align*}
P_0(Fg) = \chi_0(g) P_0(F) \  \mathrm{and}\ \  P_1(Fg) = \chi_1(g) P_1(F). 
\end{align*}
It is easy to see that 
\begin{proposition}
The pair $\left(H(\bC) , W(\bC)\right)$ is a prehomogeneous vector space.
\end{proposition}

\section{Periods over Heegner Points} 
\subsection{Shintani integral on the PVS of $2 \times 2 \times 2$ boxes.}
Let $\phi$ be a Schwartz function of $V(\bA)$ such that 
\begin{align*}
\phi(v) = e^{-  \pi \tau ( a^2 + b^2 + c^2 + d^2 + e^2 + f^2 + g^2)} \ \ \mathrm{for} \ v =  \left(  \left( \begin{matrix} a & b \\
 c & d  \end{matrix} \right), \left( \begin{matrix} e & f \\
 g & h  \end{matrix} \right) \right) \in V(\bR), 
\end{align*}
where $\tau$ is a positive real number, and 
 \begin{align*}
\phi(v) = 1 \ \ \mathrm{for} \ v \in V(\bZ_p) .
\end{align*}
Given $v \in V_D(\bQ)$, denote by $S^2_v, S^3_v$ the matrix representation of associated binary quadratic form $Q^2_v, Q^3_v$, respectively. Define the Shintani integral to be
\begin{align*}
J(f, \phi ) = \int_{Z_G^1(\bA) G^1(\bQ) \backslash G^1(\bA)} \sum_{v \in V_D(\bQ)} \phi(vg) \Phi_f^{S^2_v}(g_2) \overline{\Phi_f^{S^3_v}(g_3) } dg.
\end{align*}
Then we need to prove that
\begin{align*}
 J(f, \phi) \sim |\mathcal{P}(\Phi_f, 1)|^2.
\end{align*}

One of the ingredients in the proof is due to Bhargava's remarkable results showing that the the set of $\SL_2(\bZ) \times \SL_2(\bZ) \times \SL_2(\bZ)$-equivalent classes of projective cubes is isomorphic to the 
pairs of the strict ideal class group of a quadratic ring. In other words, if denote by $\Lambda(D)$ the set of Heegner points with discriminant $D$, then the integral orbits of $V(\bR)$ is parametrized by the pairs of Heegner points:
\begin{align*}
V(\bZ)/G^1(\bZ) \longleftrightarrow \left \{   (z_i, z_j) : z_i, z_j \in \Lambda(D) \right \}.
\end{align*}
Another ingredient in the proof is to determine the stabilizer group of the action $G^1(\bR)$. Let $S =  \left( \begin{matrix}  a & b/2 \\
 b/2 & c  \end{matrix} \right) 
$ denote the Heegner point $z$. Then the stabilizer in one of $M(\bR)$ is exactly $T^1(\bR)$, where $T^1(\bR)= T(\bR) \cap \SL_2(\bR)$ and
\begin{align*}
T^1(\bR)   
& =   \left\{\left( \begin{matrix} x + y b/2 & yc \\
 -ya & x - yb/2  \end{matrix} \right) :  x, y \in \bR , x^2 - y^2D/4 = 1\right\}.
 \end{align*}
\begin{theorem} \label{thm:periods}
The Shintani integral $J(f, \phi )$ can be evaluated as
\begin{align*}
&  J(f, \phi )  \sim   L(\phi) \cdot  \frac{4}{w(D)^2}  \left(\frac{h_K}{\mathrm{vol}(K^{\times} \bA^{\times} \backslash \bA^{\times}_K)} \right)^2  |\mathcal{P}(\Phi_f, 1)|^2 , \ \mathrm{as} \ \tau \to \infty , 
\end{align*}
where the distribution $L$ is given by 
\begin{align*}
L(\phi) =  \int_{Z_G^1(\bR) \backslash  G^1(\bR)} \phi(vg) dg,
\end{align*}
where $v$ is any element in $V_D(\bR)$.
\end{theorem}
 \begin{proof}
 Let $D$ be a negative integer that is the discriminant of the imaginary quadratic field $\bQ(\sqrt{D})$ and define
 \begin{align*}
 S = S(D) = \begin{cases}
 \left(   \begin{matrix}   
 1 & 0 \\
  0 & \frac{-D}{4}
 \end{matrix}\right) \ \  \mathrm{if} \ D \equiv 0 \ (\mathrm{mod} \ 4),\\
 \left(   \begin{matrix}   
 1 & \frac{1}{2} \\
  \frac{1}{2} & \frac{1-D}{4}
 \end{matrix}\right) \ \  \mathrm{if} \ D \equiv 1 \ (\mathrm{mod} \ 4).
 \end{cases}
 \end{align*}
 Recall that $\mathrm{Cl}_K$ denotes the ideal class group of $K$ . Let $(t_c), c\in \mathrm{Cl}_K$, be the coset representatives such that
 \begin{align*}
 T(\bA) = \sqcup_{c} t_c T(\bQ) T(\bR) \prod_{p < \infty} \left( T(\bQ_p) \cap \GL_2(\bZ_p) \right)
 \end{align*}
 with $t_c \in \prod_{p < \infty} T(\bQ_p)$.
 If we write 
 \begin{align*}
 t_c = \gamma_c m_c k_c,
 \end{align*}
 with $\gamma_c \in \GL_2(\bQ), m_c \in \GL_2^{+}(\bR)$ and $k_c \in \prod_{p < \infty} \GL_2(\bZ_p)$, then the matrices 
 \begin{align*}
 S_c = \det(\gamma_c)^{-1} ({}^t \gamma_c) S(D) \gamma_c
 \end{align*}
 form a set of representatives of the $\SL_2(\bZ)$-equivalent classes of integral binary quadratic forms of discriminant $D$. 
 Unfolding the integral, we have
 \begin{align*}
 & J(f, \phi) =  \int_{Z_G^1(\bA) G^1(\bQ) \backslash G^1(\bA)} \sum_{v \in V_D(\bQ)} \phi(vg)\Phi_f^{S^2_v}(g_2) \overline{\Phi_f^{S^3_v}(g_3) } dg \\
 & =   \int_{Z_G^1(\bR) G^1(\bZ) \backslash G^1(\bR)}  \sum_{v \in V_D(\bZ)} \phi(vg)\Phi_f^{S^2_v}(g_2) \overline{\Phi_f^{S^3_v}(g_3) } dg \\
 & = \frac{4}{ w(D)^2}\sum_{(z_i, z_j)  \in \Lambda(D) \times \Lambda(D) }  \int_{Z_G^1(\bR) \backslash  G^1(\bR)}  \phi(v_{ij}g)\Phi_f^{S^2_{v_{ij}}}(g_2) \overline{\Phi_f^{S^3_{v_{ij}}}(g_3) }  dg \\
  & = \frac{4}{w(D)^2}\sum_{(z_i, z_j)  \in \Lambda(D) \times \Lambda(D) }  \int_{ Z_G^1(\bR) \backslash  G^1(\bR)}  \phi(v_0g) \Phi_f^{S^2_{v_{ij}}}(m_ig_2) \overline{\Phi_f^{S^3_{v_{ij}}}(m_jg_3) }  dg \\
    & = \frac{4}{w(D)^2}\sum_{(z_i, z_j)  \in \Lambda(D) \times \Lambda(D) }  \int_{ Z_G^1(\bR) \backslash  G^1(\bR)}  \phi(v_0g) \Phi_f^{S}(t_ig_2) \overline{\Phi_f^{S}(t_jg_3) }  dg ,
 \end{align*}
 where $w(D)$ is the counting measure of stabilizer inside each copy of $\SL_2(\bZ)$. Notice that
 \begin{align*}
   &  \sum_{(z_i, z_j)  \in \Lambda(D) \times \Lambda(D) }  \int_{T^1(\bR)}   \Phi_f^{S}(t_i t_{\infty} g_2) dt_{\infty}   \int_{T^1(\bR)}  \overline{\Phi_f^{S}(t_j t_{\infty}g_3) } dt_{\infty}  \\
   & =  \sum_{c \in \mathrm{Cl}_K}  \int_{T^1(\bR)}  \Phi_f^{S}(t_c t_{\infty} g_2) dt_{\infty}  \cdot  \sum_{c \in \mathrm{Cl}_K}   \int_{T^1(\bR)}  \overline{\Phi_f^{S}(t_c t_{\infty}g_3)}  dt_{\infty}  \\
 & = \left(\frac{h_K}{\mathrm{vol}(K^{\times} \bA^{\times} \backslash \bA^{\times}_K)} \right)^2  
 \int_{\bA^{\times} T(\bQ) \backslash  T(\bA)}   \Phi_f^{S}( t g_2)  dt \cdot    \int_{\bA^{\times} T(\bQ) \backslash T(\bA)}  \overline{\Phi_f^{S}(t g_3)}  dt,
\end{align*}
assuming $T^1(\bR)$ has measure 1. 
We will set $C(g)$ such that
\begin{align*}
 & \int_{\bA^{\times} T(\bQ) \backslash  T(\bA)}  \int_{U(\bQ) \backslash U(\bA)} \Phi_f(u t g) \theta_S(u)^{-1}   du \\
 & = C(g) \int_{\bA^{\times} T(\bQ) \backslash  T(\bA)}  \int_{U(\bQ) \backslash U(\bA)} \Phi_f(u t ) \theta_S(u)^{-1} du  \\
 & = C(g) \mathcal{P}(\Phi_f, 1). 
\end{align*}
Note that $C(1) = 1$. To avoid evaluating the integral, we take the asymptotic as $\tau \to \infty$,
\begin{align*}
& \int_{Z_G^1(\bR) \backslash  G^1(\bR)}  \phi_0(v_0 g) C(g_2)  \overline{C(g_3)} dg \sim \int_{ Z_G^1(\bR) \backslash  G^1(\bR)}  \phi(v_0 g)   dg,  \ \mathrm{as} \ \tau \to \infty,
\end{align*}
then the proof follows.
 \end{proof}
 
In next section we will obtain the same period integral by defining a similar Shintani integral on a larger PVS, which is the skew-symmetrization of PVS of $2 \times 2 \times 2$ boxes.
 
\subsection{Shintani integral on the PVS of pairs of quaternary alternating $2$-forms}
In this section we construct a Shintani integral for the PVS of pairs of quaternary alternating $2$-forms. Let $H(\bQ)$ denote the subgroup $B_2(\bQ) \times P_{2,2}(\bQ)$, 
\begin{align*}
 g  = (g_1, p) =  \left(  \left( \begin{matrix}  a_1 & b_1 \\
0 & d_1 & \\
   \end{matrix} \right),  
 \left( \begin{matrix}  a_2 & b_2 & s & t\\
c_2& d_2 & u & v\\
0 & 0 & a_3 & b_3 \\
0 & 0 & c_3 &d_3  \end{matrix} \right) \right) \ \ \mathrm{for} \ g \in H(\bQ).
\end{align*} 
Denote by $H^1(\bQ)$ the subgroup of $H(\bQ)$ defined by 
\begin{align*}
\{ g=(g_1, p) : \det(g_1) \det(p) = 1\},
\end{align*}
Similarly, let $B(\bQ)$ denote the subgroup $B_2(\bQ) \times P_{\mathrm{min}}(\bQ)$, 
\begin{align*}
 g  = (g_1, p) =  \left(  \left( \begin{matrix}  a_1 & b_1 \\
0 & d_1 & \\
   \end{matrix} \right),  
 \left( \begin{matrix}  a_2 & b_2 & s & t\\
0 & d_2 & u & v\\
0 & 0 & a_3 & b_3 \\
0 & 0 & 0 &d_3  \end{matrix} \right) \right) \ \ \mathrm{for} \ g \in H(\bQ).
\end{align*} 
Denote by $B^1(\bQ)$ the subgroup of $B(\bQ)$ defined by 
\begin{align*}
\{ g=(g_1, p) : \det(g_1) \det(p) = 1\},
\end{align*}
Let $Z^1_H(\bQ)$ be the center of the group $H^1(\bQ)$ and $T_H(\bQ)$ be the subgroup acting on $W(\bQ)$ trivially. Then 
\begin{align*}
T_H(\bQ) = \left\{ \left(  \left( \begin{matrix}  t^{-2} & 0 \\
0 & t^{-2} & \\
   \end{matrix} \right),  
 \left( \begin{matrix}  t & 0 & 0 & 0\\
0 & t & 0 & 0\\
0 & 0 & t & 0 \\
0 & 0 & 0 & t  \end{matrix} \right) \right)  : t \in \bQ^{\times} \right \} ,
\end{align*}
and we will identify 
\begin{align*}
t \in \bQ^{\times}  \mapsto  \left(  \left( \begin{matrix}  t^{-1} & 0 \\
0 & t^{-1} & \\
   \end{matrix} \right),  
 \left( \begin{matrix}  t & 0 & 0 & 0\\
0 & t & 0 & 0\\
0 & 0 & 1 & 0 \\
0 & 0 & 0 & 1 \end{matrix} \right) \right)  \in  T_H(\bQ) \backslash Z^1_H(\bQ) . 
\end{align*}
Let $M(\bQ)$ denote the Levi subgroup of $H(\bQ)$ by
\begin{align*}
g = (g_1, g_2, g_3) =  \left( g_1,  \left(\begin{matrix} g_2 & 0\\
0 & d_3\\
  \end{matrix} \right) \right) \ \ \mathrm{for} \ g \in M(\bQ),
\end{align*}
and also denote by  $U(\bQ)$ the unipotent subgroup generated by
\begin{align*}
\left(  \begin{matrix}  1 &0\\
0 &  1   \end{matrix}\right)  \times \left(\begin{matrix} 1 & 0 & s& t\\
0 & 1 & u & v\\
0 & 0 & 1  & 0 \\
0 & 0 & 0 & 1     \end{matrix} \right), \ \ s, t, u, v \in \bQ.
\end{align*}
We are also free to identify 
\begin{align*}
g_ 2 =  \left(\begin{matrix} g_2 & 0\\
0 & \det(g_2) \cdot {}^t g_2^{-1}\\
  \end{matrix} \right) , \ \ g_ 3 =  \left(\begin{matrix} g_3 & 0\\
0 & \det(g_3) \cdot {}^t g_3^{-1}\\
  \end{matrix} \right) \in \mathrm{GSp}_4
\end{align*}
as before.

Let $H(\bA), B(\bA)$ be the adelization of $H(\bQ)$.  Let $W_D(\bA)$ be the subset of all alternating $2$-forms with discriminant $D$ and with vanishing entry $r_1$, and $W^0_D(\bA)$ be the subset of $W_D(\bA)$ with vanishing entry $a$. Then $H^1(\bA), B^1(\bA)$ acts invariantly on $W_D(\bA)$ and $W^0_D(\bA)$ respectively.
Recall two rational characters $\chi_0$ and $\chi_1$ of $H(\bA)$ are given by 
\begin{align*}
P_0( wg) = \chi_0(g) P_0(w),\ \ \mathrm{and} \ \  P_1(wg) = \chi_1(g) P_1(w),
\end{align*}
where $P_0(w) = r_2(w)$ and $P_1(w) = -\mathrm{Pfaff}(M_w)$.
The measure on $H(\bA)$ is defined as follows. The measure on $B_2(\bA)$ is the usual left invariant measure. The measure on $P_{2,2}(\bA)$ is given by $dp = dm dn$ if we decompose $p \in P_{2,2}(\bA)$ as
\begin{align*}
 p=m n = \left( \begin{matrix}  a_2 & b_2 & 0 & 0  \\
c_2 & d_2 & 0 & 0\\
0 & 0 & a_3 & b_3 \\
0 & 0 & c_3 &d_3   \end{matrix} \right)
  \left(  \begin{matrix}  1 & 0& s & t\\
0 & 1 & u & v\\
0 & 0 & 1& 0\\
0 & 0 & 0 & 1 \end{matrix} \right),
\end{align*}
where $dm$ is the left invariant measure of the Levi subgroup $M(\bA)$. 

Now we are going to define a Shintani integral associated to the PVS $(H^1(\bA), W_D(\bA))$.
Given any $n \neq 0 \in \bZ$, we define a family of Schwartz functions $\phi_{n}$ on $W(\bA)$ such that
\begin{align*}
\phi_{n} (w  ) & = \exp^{- \pi \tau\left(a(w  )^2 + b(w  )^2+ c(w  ) ^2 + d(w  )^2 + e(w  )^2 + f(w  )^2+ g(w  ) ^2 + h(w  )^2\right)}   \\
& \ \ \ \cdot \exp^{ -\pi  \left(n^2 r_1(w)^2 + r_2(w)^2 + l_1(w)^2+ l_2(w)^2  \right)},
\end{align*}
where $\tau$ is again a positive real parameter, for
\begin{align*}
w = \left( \left( \begin{matrix}
 0 & r_1 & a & b \\
 -r_1 & 0 & c & d \\
 -a & -c & 0 & l_1 \\
 -b & -d & -l_1 & 0 
\end{matrix} \right),  \left( \begin{matrix}
 0 & r_2 & e & f \\
 -r_2 & 0 & g & h \\
 -e & -g & 0 & l_2 \\
 -f & -h & -l_2 & 0 
\end{matrix} \right) \right) \in W_D(\bR),
\end{align*}
and 
\begin{align*}
\phi_n(w  ) = 1 \ \ \mathrm{for} \ w \in W_D^0(\bZ_p), |r_1(w)|_p = |b(w)|_p = 1. 
\end{align*}
Define
\begin{align*}
\psi (w  ) & = \sum_{n \neq 0} \phi_n(w). 
\end{align*}
Consider the embedding:
$$V_D(\bQ) \to W_D(\bQ)$$
defined by
\begin{align*}
\left( \left(  \begin{matrix} a &b\\
c & d   \end{matrix}\right) , \left(\begin{matrix} e &f \\
g & h  \end{matrix} \right) \right)  \mapsto
\left( \left(  \begin{matrix} 0 & 0 & a& b\\
0 & 0 & c & d\\
-a & -c & 0  & 0 \\
-b & -d & 0 & 0    \end{matrix}\right) , \left(\begin{matrix} 0 & 1 & e& f\\
-1 & 0 & g & h\\
-e & -g & 0  & 0 \\
-f & -h & 0 & 0   \end{matrix} \right) \right).
\end{align*}
Let $\hat{V}_D(\bQ)$ be the image of $V_D(\bQ)$ under this embedding. Define a mixed Eisenstein series
\begin{align*}
\Psi^0(s_0, s_1, g) = \sum_w \sum_{\gamma \in  T_H(\bQ) \backslash B^1(\bQ)} |\chi_0(g)|^{s_0} |\chi_1(g)|^{s_1}  \psi (w\gamma g)  \Phi_f^{S^2_v}(\gamma g_2) \overline{\Phi_f^{S^3_v}( \gamma g_3) },
\end{align*}
where $w$ is the images of $v$, the representative of $\SL_2(\bZ) \times \SL_2(\bZ) \times \SL_2(\bZ)$-orbit of $V_D(\bZ)$, chosen such that $w \in W^0_D(\bQ)$. 
Define a Shintani integral to be
\begin{align*}
&J^0(s_0, s_1, f, \psi) = \int_{B^1(\bQ) \backslash B^1(\bA) /T_H(\bA) } \Psi^0(s_0, s_1, g) dg.
\end{align*} 

\subsection{Application of the Poisson summation formula.}
We compute its residue at some special values. The next proposition shows the residue turns out to be exactly the square of sums over Heegner points again. 
\begin{theorem} \label{thm:residue}
The Shintani integral $J^0(s_0, s_1, f , \psi)$ has a simple pole at $s_0 =2$ and $s_1 = 1$ with the residue gives rise to the  square product of periods over Heegner points. More precisely, 
\begin{align*}
&  \mathrm{Res}_{s_0= 2, s_1 = 1} J^0(s_0, s_1,f, \psi) =  - \frac{c(J^0, J)}{ \zeta(2)} J(f, \phi),
\end{align*}
where $c(J^0, J) = \frac{2w^2(D)}{\pi^2}$ is a measure constant. 
\end{theorem} 

\begin{proof}
Observe that 
\begin{align*}
&\int_{B^1(\bQ) \backslash B^1(\bA) /T_H(\bA)} \Psi^0(s_0, s_1, g) dg \\
& =  \frac{w^2(D)}{\pi^2} \cdot 2 \int_{y>0} \int
 \sum_{w \in \hat{V}_D(\bZ)}  y^{s_0-2} |\chi_1(g)|^{s_1} \psi(w g y)  \Phi_f^{S^2_v}(g_2) \overline{\Phi_f^{S^3_v}( g_3)}  dg d^{\times}y,
\end{align*}
where the integral is over the group 
$$\left( B_2(\bZ) \times \SL_2(\bZ) \times \SL_2(\bZ) \right) \backslash \left( B_2^{+}(\bR) \times \SL_2(\bR) \times \SL_2(\bR) \right).$$
Applying the classical Poisson summation formula, the integral has a simple pole at $s_0 = 2$ with residue
\begin{align*}
&- \int
\sum_{v \in V_D(\bZ)} |\chi_1(g)|^{s_1} \phi(v g ) \Phi_f^{S^2_v}(g_2) \overline{\Phi_f^{S^3_v}( g_3)}  dg .\\
\end{align*}
With the Eisenstein series, the above can be written as
\begin{align*}
&- \frac{1}{ \pi^{-(s_1-1)} \Gamma(s_1) \zeta(2s_1)} \cdot 2 \int_{[G^1(\bR)]} \sum_{v \in V_D(\bZ)} E^{\ast}(s_1, g_1)  \phi(v g)   \Phi_f^{S^2_v}(g_2) \overline{\Phi_f^{S^3_v}( g_3)}   dg, 
\end{align*}
where $[G^1(\bR)] = Z^1_G(\bR) G^1(\bZ) \backslash G^1(\bR)$. 
The Eisenstein series has residue $\frac{1}{2}$ at $s_1= 1$. Therefore, taking the residue at $s_1 = 1$, the above becomes
\begin{align*}
& \frac{-1}{ 2 \Gamma(1) \zeta(2)}  \int_{[G^1(\bR)]} \sum_{v \in V_D(\bZ)}   \phi(v g)   \Phi_f^{S^2_v}(g_2) \overline{\Phi_f^{S^3_v}( g_3)}  dg. 
\end{align*}
\end{proof}

By means of the functional equation of $E^{\ast}(s_1, g_1)$, we have
\begin{corollary} \label{corollary:residue}
The residue of Shintani integral $J^0(s_0, s_1, f, \psi)$ at $s_0=2$ can be expressed by 
\begin{align*}
\mathrm{Res}_{s_0=2} J^0(s_0, s_1, f, \psi) =\frac{ \pi^{s_1} \Gamma(1-s_1) \zeta(2(1-s_1))}{ \pi^{-(s_1-1)} \Gamma(s_1) \zeta(2s_1)} \mathrm{Res}_{s_0 =2}J^0(s_0, 1-s_1, f, \psi).
\end{align*}
\end{corollary}

\section{Central value of L-function}
\subsection{Invariant Functional}
Recall that we may associate to a pair of elements $f \in \pi$ and $\tilde{f} \in \tilde{\pi}$ a smooth function $\Phi_{f,\tilde{f}}(g) = \langle\pi(g)f, \tilde{f}\rangle$ on $G$, called a matrix coefficient of $\pi$.  The following integral 
\begin{align*}
\int_{\bA^{\times} \backslash T(\bA)}  \int_{U(\bA)} \Phi_{f,\tilde{f}}(gu) \theta_S^{-1}(u) du dg
\end{align*}
decomposes into an Euler factorization with 
\begin{align*}
\int_{\bQ_v^{\times} \backslash T(\bQ_v)}  \int_{U_v} \Phi_{f,\tilde{f}}(gu) \theta_S^{-1}(u) du dg \sim  \frac{\zeta_v(2) \zeta_v(4) L_v\left( \frac{1}{2}, \pi  \right)}{L_v\left( 1, \pi , \mathrm{Ad}\right) L_v(1, \chi_{K/F})  }.
\end{align*}
Therefore, in order to prove the conjecture, we need to prove that 
\begin{align*}
\int_{\bA^{\times} \backslash T(\bA)}  \int_{U(\bA)} \Phi_{f,\tilde{f}}(gu) \theta_S^{-1}(u) du dg  \sim |\mathcal{P}(f, 1)|^2.
\end{align*}

\begin{theorem}
The residue of Shintani integral $J^0(s_0, s_1, f, \psi)$ at $s_0=2$ and $s_1 = 1$ can be evaluated by 
\begin{align*}
- \frac{ L(\phi)}{\pi^2} \cdot  C  \int_{\bA^{\times} \backslash T(\bA)}  \int_{U(\bA)} \Phi_{f,\tilde{f}}(gu) \theta_S^{-1}(u) du dg, \ \mathrm{as} \ \tau \to \infty, 
\end{align*}
where the constant 
\begin{align*}
C = \frac{||f||^2  \mathrm{vol}(K^{\times} \bA^{\times} \backslash \bA_{K}^{\times})}{\zeta(2)} .
\end{align*}
\end{theorem}
\begin{proof}
The key observation is that the action $B^1(\bQ)$ on $W^0_D(\bQ)$ has a single orbit represented by $w$
\begin{align*}
 \left( \left(   \begin{matrix}  0 & 0 & 0 & 1 \\
0 & 0 & -1 & 0  \\
0 & 1 & 0 & 0 \\
-1 & 0 & 0& 0\\
  \end{matrix} \right),  \left(   \begin{matrix}  0 & 1 & 1 & 0 \\
-1 & 0 & 0 & -\frac{D}{4}  \\
0 & 0 & 0 & 0 \\
0 & 0 & 0 & 0\\
  \end{matrix} \right) \right) \ \ \mathrm{if} \  D \equiv 0 \ (\mathrm{mod} \ 4), 
\end{align*}
or 
\begin{align*}
 \left( \left(   \begin{matrix}  0 & 0 & 0 & 1 \\
0 & 0 & -1 & -1  \\
0 & 1 & 0 & 0 \\
-1 & 0 & 0& 0\\
  \end{matrix} \right),  \left(   \begin{matrix}  0 & 1 & 1 & 1 \\
-1 & 0 & -1 & -\frac{(D+3)}{4}  \\
0 & 0 & 0 & 0 \\
0 & 0 & 0 & 0\\
  \end{matrix} \right) \right) \ \ \mathrm{if} \  D \equiv 1 \ (\mathrm{mod} \ 4). 
  \end{align*}
Unfolding the sum over rational orbits in $W^0_D(\bQ)$, it reduces to a single orbital integral
\begin{align*}
\int_{ T_H(\bA) \backslash B^1(\bA)}  |\chi_0(g)|^{s_0} |\chi_1(g)|^{1-s_1}  \psi (w g)  \Phi_f^{S}( g_2) \overline{\Phi_f^{S}( g_3) }dg.
\end{align*}
Another key point is the stabilizer group of $w$ in $B^1(\bA)$ is generated by
\begin{align*}
 \left( \begin{matrix}  a_3 & 0 \\
0 & a_3^{-1} \end{matrix} \right) \times \pm  \left( \begin{matrix} a_3 & b_3 & a_3 b_3 & b_3^2 - \frac{D}{4} \\
0 & 1  & -a_3 & -b_3 \\
0 & 0 & a_3 & b_3 \\
0& 0 & 0 & 1 \end{matrix} \right) \ \ \mathrm{if} \  D \equiv 0 \ (\mathrm{mod} \ 4), 
\end{align*}
or 
\begin{align*}
 \left( \begin{matrix}  a_3 & 0 \\
0 & a_3^{-1} \end{matrix} \right) \times \pm  \left( \begin{matrix} a_3 & b_3 & - a_3 + a_3 b_3 & b_3^2 - \frac{(D+3)}{4} \\
0 & 1  & -a_3 & -1-b_3 \\
0 & 0 & a_3 & b_3 \\
0& 0 & 0 & 1 \end{matrix} \right) \ \ \mathrm{if} \  D \equiv 1 \ (\mathrm{mod} \ 4), 
\end{align*}
where $a_3, b_3 \in \bA$, so it is isomorphic to the Borel subgroup $B_2(\bA)$ of $\GL_2(\bA)$.  
Therefore, integrating over stabilizer group, the integral valued at $s_1 = 1$ becomes
\begin{align*}
 \int_{ B_2(\bA) \backslash B^1(\bA)}  |\chi_0(g)|^{s_0}  \psi (w g)  \int_{\bA^{\times} \backslash B_2(\bA)} \Phi_f^{S}( g_3 g_2) \overline{\Phi_f^{S}( g_3) } dg_3 dg.
\end{align*}
Now at finite place the integral reduces to the orbit counting. Write
\begin{align*}
\left(  \left(   \begin{matrix}  0 & 0 & 0 & y_1  \\
0 & 0 & y_1y_2  & d  \\
0 & -y_1y_2  & 0 & 0 \\
-y_1 & -d & 0& 0\\
  \end{matrix} \right),  \left(   \begin{matrix}  0 & t & e  & f \\
-t & 0 & g& h  \\
-e& -g& 0 & 0 \\
-f & -h & 0& 0\\
  \end{matrix} \right) \right).
\end{align*}
Let us ask the question: how many orbits in $W^0_D(\bZ_p)$
such that $t \in \bZ_p^{\times}$, $y_1 \in \bZ_p^{\times}$ and $ y_1y_2 \in \bZ_p$.
The answer is: 
the number of orbits is equal to the number of solutions of 
$x$ such that $x^2 = D \ \mathrm{mod} \ 4|y_2|_p^{-1}$. Therefore, after adding local orbit integrals, the integral becomes  
\begin{align*}
 \int_{ B_2(\bA) \backslash B^1(\bA)}  |y_{\infty}|^{s_0}  \psi_0 (w g)  \int_{\bA^{\times} \backslash B_2(\bA)} \int_{\bA^{\times} \backslash T(\bA)} \Phi_f^{S}( g_3 t) \overline{\Phi_f^{S}( g_3) } dg_3 dg
\end{align*}
dividing the measure constant $\left(\frac{\mathrm{vol}(K^{\times} \bA^{\times} \backslash \bA_{K}^{\times})}{h_{K}}\right)^2$. Applying the classical Poisson summation formula, the integral has a simple pole at $s_0 = 2$ with residue
\begin{align*}
\frac{1}{2} \cdot \frac{8 L(\phi)}{4 \pi^3 } \cdot 2 \int_{\bA^{\times} \backslash B_2(\bA)} \int_{\bA^{\times} \backslash T(\bA)} \Phi_f^{S}( g_3 t) \overline{\Phi_f^{S}( g_3) } dg_3 dg, 
\end{align*}
as $\tau \to \infty$. 
The rest of proof is showing that 
\begin{align*}
 \int_{\bA^{\times} \backslash B_2(\bA)} \int_{\bA^{\times} \backslash T(\bA)} \Phi_f^{S}( g_3 t) \overline{\Phi_f^{S}( g_3) } dt dg_3 
 \end{align*}
 is equal to
 \begin{align*}
\int_{\bA^{\times} \backslash T(\bA)}  \int_{U(\bA)} \Phi_{f,\tilde{f}}(tu) \theta_S^{-1}(u) du dt
 \end{align*}
 up to a specified constant.
 First the above integral can be written as 
 \begin{align*}
 \int_{U(\bA)} \Phi_{f,\tilde{f}}(tu) \theta_S^{-1}(u) du =  c \cdot \int_{T(\bA) \backslash \GL_2(\bA)} \int_{U(\bA)} \Phi_f^{S}( g t u) \overline{\Phi_f^{S}( g) }  \theta_S^{-1}(u)   du,
 \end{align*}
 where
 \begin{align*}
c= \mathrm{vol}(K^{\times} \bA^{\times} \backslash \bA_{K}^{\times})  \frac{\mathrm{vol}(Z(\bA) \mathrm{GSp}_4(\bQ) \backslash \mathrm{GSp}_4(\bA))}{\mathrm{vol}(Z(\bA) \GL_2(\bQ) \backslash \GL_2(\bA))} = \mathrm{vol}(K^{\times} \bA^{\times} \backslash \bA_{K}^{\times}) \zeta(4). 
 \end{align*}
Let $f_t$ be the function on $\mathrm{M}_2^{\mathrm{sym}}$ defined by 
\begin{align*}
f_t(\xi g) = \int_{ \bA^{\times} \backslash B_2(\bA)} \Phi_f^{S}( p g  t) \overline{\Phi_f^{S}( p g) } dp,
\end{align*}
where $\xi = S$.
Extending by $0$, we view $f_t$ as a function in $C_c^{\infty}(\mathrm{M}_2^{\mathrm{sym}})$. We write $\langle\cdot, \cdot \rangle$ for the correspondence pairing on $\mathrm{M}_2^{\mathrm{sym}} \times \mathrm{M}_2^{\mathrm{sym}}$. Then
\begin{align*}
&  \int_{  T(\bA) \backslash  \GL_2(\bA)} \int_{U(\bA)} \Phi_f^{S}( g t u) \overline{\Phi_f^{S}( g ) }  \theta_S^{-1}(u)   du 
\end{align*}
is equal to 
\begin{align*}
  & \int  f_t(\xi g) \theta\langle \xi g, u \rangle dg \theta\langle \xi, u \rangle^{-1} du, 
\end{align*}
which by Fourier inversion is
\begin{align*}
f_t(\xi ) = \int_{ \bA^{\times} \backslash B_2(\bA)} \Phi_f^{S}( p  t) \overline{\Phi_f^{S}( p ) } dp.
\end{align*}
This completes the proof of the theorem.
\end{proof}

\subsection{Global L-function}
In the last proof we have shown that two global integrals are related
\begin{align*}
\int_{\bA^{\times} \backslash B_2(\bA)} \int_{\bA^{\times} \backslash T(\bA)} \Phi_f^{S}( g_3 t) \overline{\Phi_f^{S}( g_3) } dt dg_3 \sim 
\int_{\bA^{\times} \backslash T(\bA)}  \int_{U(\bA)} \Phi_{f,\tilde{f}}(tu) \theta_S^{-1}(u) du 
\end{align*}
up to a constant. 
We have to show the integral 
\begin{align*}
\int_{\bA^{\times} \backslash T(\bA)}  \int_{U(\bA)} \Phi_{f,\tilde{f}}(tu) \theta_S^{-1}(u) du 
\end{align*}
is related to the expected central value of L-function. To this end, we consider a family of integrals 
\begin{align*}
\int_{\bA^{\times} \backslash B_2(\bA)} \int_{\bA^{\times} \backslash T(\bA)}  \lambda_s(t) \Phi_f^{S}( g_3 t) \overline{\Phi_f^{S}( g_3) } dt dg_3,
\end{align*}
where the continuous homomorphism from $B_2(\bA)$ to $\bC^{\times}$ is defined as follows
\begin{align*}
\lambda_s(b): \left( \begin{matrix}
b_1& \ast \\
0 & b_2
\end{matrix} \right) \mapsto |b_1|^s |b_2|^{-s},
\end{align*}
and extends to $\GL_2(\bA)$. 
Now we explain how this deformation of integrals affects the Fourier coefficients of matrix coefficients. Let us consider the unramified situation. Choose a place $v$ of $F$ that satisfies six conditions prior to the \cite[Theorem 2.2]{liu}. Such places are called $\emph{good}$. 
Let $\pi_v =\mathrm{I}(\Xi) =  \mathrm{Ind}_B^G(\Xi)$ be an unramified principal series of $G = \mathrm{PGSp}_4$. Let $\Phi_{\Xi} = \Phi_{f_v, \tilde{f}_v}$ be the spherical matrix coefficient of $\pi_v$, where $f_v \in \pi_v$
is the spherical vector. Denote $\Xi_s = |\det|^s \cdot \Xi$. Then our proof in last section shows that 
\begin{align*}
\int_{\bA^{\times} \backslash B_2(\bA)} \int_{\bA^{\times} \backslash T(\bA)}  \lambda_s(t) \Phi_f^{S}( g_3 t) \overline{\Phi_f^{S}( g_3) } dt dg_3
\end{align*}
is equal to
\begin{align*}
\int_{\bA^{\times} \backslash T(\bA)}  \int_{U(\bA)} \Phi_{f_s,\tilde{f}_s}(tu) \theta_S^{-1}(u) du 
\end{align*}
up to the same constant, where $f_{s, v} \in  \mathrm{Ind}_{B_v}^{G_v}(\Xi_s)$ at good places. Assume the real part of $s$, $\mathrm{Re}(s)$, is large. Then the proof in \cite[Section 3.3, Appendices]{liu} shows that 
\begin{align*}
\int_{\bA^{\times} \backslash T(\bA)}  \int_{U(\bA)} \Phi_{f_s,\tilde{f}_s}(tu) \theta_S^{-1}(u) du 
\end{align*} 
is eulerian, and is equal to
\begin{align*}
\frac{\zeta(2) \zeta(4)L(s + \frac{1}{2}, \pi)}{L(2s +1,  \pi, \mathrm{Ad}) L(1, \chi_{K/F})} \prod_v \left(\frac{\zeta_v(2) \zeta_v(4)L_v(s + \frac{1}{2}, \pi)}{L_v(2s +1, \pi, \mathrm{Ad}) L_v(1, \chi_{K/F})} \right)^{-1}\alpha(f_{v,s}, \tilde{f}_{v, s}),
\end{align*} 
where 
\begin{align*}
\alpha(f_{v,s}, \tilde{f}_{v, s}) = \int_{\bQ_v^{\times} \backslash T(\bQ_v)}  \int_{U(\bQ_v)} \Phi_{f_{v,s},\tilde{f}_{v,s}}(tu) \theta_S^{-1}(u) du.
\end{align*}
Those local $L$-values follow from \cite[Theorem 4.5]{GPSR}.
Therefore the value at $s = 0$ produces the expected central value of $L$-function.

\bibliography{BCWCV}

\end{document}